\documentclass[12pt]{article}

\usepackage[T1]{fontenc}
\usepackage[utf8]{inputenc}
\usepackage{amsmath,amssymb,amsfonts}
\usepackage{amsthm}
\usepackage{authblk}
\usepackage[numbers]{natbib}
\usepackage{enumitem}
\usepackage{geometry}
\usepackage{hyperref}
\usepackage{color}

\newtheorem{theorem}{Theorem}[section]
\newtheorem{lemma}[theorem]{Lemma}

\newtheorem{definition}[theorem]{Definition}

\newtheorem{conjecture}[theorem]{Conjecture}

\newcommand{\Hcal}{\mathcal H}

\title{A stability theorem for Berge Hamiltonian cycles under a minimum degree condition}

\author[1]{\small\bf Yichen Wang\thanks{Email: wangyich22@mails.tsinghua.edu.cn}}
\author[2]{\small\bf D\'aniel Gerbner\thanks{Email: gerbner.daniel@renyi.hu
}}
\author[1]{\small\bf Xiamiao Zhao\thanks{\textit{Corresponding author:}Email: zxm23@mails.tsinghua.edu.cn}}

\affil[1]{\small Department of Mathematical Sciences, Tsinghua University, Beijing, P.R. China.}
\affil[2]{\small HUN-REN Alfr\'ed R\'enyi Institute of Mathematics, Budapest, Hungary.}

\date{}

\begin{document}

\maketitle

\begin{abstract}
In this paper,
we study extremal and stability problems for Berge Hamiltonian cycles in
$r$-uniform hypergraphs under a minimum degree condition.  Let
       $ g_r(n,t)=\binom{n-t}{r}+t\binom{t}{r-1}$,
and let $t=t(k)$ be the unique integer satisfying
$\binom{t-1}{r-1}<k\le \binom{t}{r-1}$.  Using a sharp
P\'osa-type degree sequence theorem
of Salia, we prove an extremal upper bound on the number of hyperedges in
an $n$-vertex $r$-uniform hypergraph with minimum degree at least $k$ and
with no Berge Hamiltonian cycle.  We also prove a stability theorem in the
dense range before the first minimizer of $g_r(n,t)$: every near-extremal
example is contained in one of two natural non-Hamiltonian constructions.
\end{abstract}

\section{Introduction}

We use the following basic notation from the beginning.  For a positive
integer $m$, write $[m]=\{1,\ldots,m\}$, and for a set $X$, write
$\binom{X}{s}$ for the family of all $s$-subsets of $X$.  For a graph or
hypergraph $F$, let $e(F)$ denote its number of edges and let
$\delta(F)$ denote its minimum vertex degree.

A Berge cycle of length $\ell$ in a hypergraph $\Hcal$ is an alternating
sequence
$$
        v_1,E_1,v_2,E_2,\ldots,v_\ell,E_\ell
$$
of distinct vertices $v_1,\ldots,v_\ell$ and distinct hyperedges
$E_1,\ldots,E_\ell$ such that $v_i,v_{i+1}\in E_i$ for every
$i\in[\ell]$, where the indices are taken modulo $\ell$.  A Berge
Hamiltonian cycle is a Berge cycle whose defining vertices are all the
vertices of the hypergraph.

Hamiltonicity in graphs is governed by classical degree conditions such as
Dirac's theorem \cite{Dirac_1952}, Ore's theorem \cite{Ore_1959},
P\'osa's theorem \cite{Posa_1962}, and Chv\'atal's degree-sequence
theorem \cite{Chvatal_1972}.  Another closely related line of work asks
for the maximum number of edges in a non-Hamiltonian graph with prescribed
minimum degree.  A theorem of Erd\H{o}s \cite{Erdos_1962_Remarks} states
that if $G$ is an $n$-vertex non-Hamiltonian graph with
$\delta(G)\ge d$, where $1\le d\le \lfloor(n-1)/2\rfloor$, then
$$
        e(G)\le
        \max\left\{
        \binom{n-d}{2}+d^2,\,
        \binom{n-\lfloor(n-1)/2\rfloor}{2}
        +\left\lfloor\frac{n-1}{2}\right\rfloor^2
        \right\}.
$$
The first extremal construction is obtained from a clique on $n-d$
vertices by adding $d$ vertices with the same $d$ neighbors in the
clique, and the second is 
%the balanced endpoint obstruction.  
obtained from a complete bipartite graph $K_{\left\lfloor\frac{n-1}{2}\right\rfloor,\left\lfloor\frac{n-1}{2}\right\rfloor}$ by adding one or two vertices so that altogether we have $n$ vertices, taking the union of these vertices and one of the parts, and adding all the edges that are inside this union.
A stability
version due to F\"uredi, Kostochka and Luo
\cite{Furedi_Kostochka_Luo_2017} shows that, above the next extremal
value, a non-Hamiltonian graph is forced to be contained in one of the
corresponding canonical obstructions.  The present paper develops an
analogue of this extremal and stability picture for Berge Hamiltonian
cycles in uniform hypergraphs.

In hypergraphs, Berge Hamiltonicity has been studied from several
directions, including extremal results for long Berge cycles
\cite{Furedi_Kostochka_Luo_2019,Furedi_Kostochka_Luo_2020,Furedi_Kostochka_Luo_2021}
and Dirac-type theorems \cite{Kostochka_Luo_McCourt_2024}.  Our main tool
is Salia's P\'osa-type theorem \cite{Salia_2024}, which we record in
Section~\ref{sec:preliminaries}.

The natural extremal examples for our problem are the following two
hypergraphs.  Let $T$ and $S$ be disjoint sets with $|T|=|S|=t$, and let
$A$ be a set of size $n-t$ disjoint from $T$ and containing $S$.

\begin{definition}
The hypergraph $\Hcal^{(1)}_{n,t}$ has vertex set $T\cup A$.  It contains
all $r$-sets inside $A$, and all $r$-sets of the form
$$
        \{x\}\cup R,\qquad x\in T,\quad R\in\binom{S}{r-1}.
$$
\end{definition}

\begin{definition}
The hypergraph $\Hcal^{(2)}_{n,t}$ is the union of two complete
$r$-uniform hypergraphs whose vertex sets have sizes $n-t$ and $t+1$ and
intersect in exactly one vertex.
\end{definition}
For integers $n$ and $t$, put
\begin{equation}\label{eq:def-gr}
        g_r(n,t)=\binom{n-t}{r}+t\binom{t}{r-1}.
\end{equation}
This is the number of edges in $\Hcal^{(1)}_{n,t}$.  Given a
minimum-degree parameter $k$, we write $t=t(k)$ for the unique integer
satisfying
$$
        \binom{t-1}{r-1}<k\le \binom{t}{r-1}.
$$

Both constructions have no Berge Hamiltonian cycle.  For
$\Hcal^{(2)}_{n,t}$ this follows from the cut vertex.  For
$\Hcal^{(1)}_{n,t}$, every defining edge incident with a vertex of $T$
must use a vertex of $S$.  Thus, in a Hamiltonian cycle
all the vertices adjacent to the vertices of $S$ belong to $T$, thus the cycle alternates between $S$ and $T$,
%would use all degree of every vertex of $S$ to pass through the vertices of $T$, 
leaving no
place to insert the vertices of $A\setminus S$, which is nonempty in the
range $t\le \lfloor (n-1)/2\rfloor$.

Our first result is the extremal upper bound.

\begin{theorem}[Extremal bound]\label{thm:extremal}
Let $r\ge 3$, let $n>2r$, and let $\Hcal$ be an $n$-vertex
$r$-uniform hypergraph with no Berge Hamiltonian cycle.  Let
$$
        r+1\le k< \binom{\lfloor (n-1)/2\rfloor}{r-1},
$$
and let $t=t(k)$ be the unique integer satisfying
$$
        \binom{t-1}{r-1}< k\le \binom{t}{r-1}.
$$
If $\delta(\Hcal)\ge k$, then
$$
e(\Hcal)\le
\begin{cases}
\max\left\{g_r(n,t),g_r\left(n,\left\lfloor\frac{n-1}{2}\right\rfloor\right)\right\},
& n\text{ odd},\\[4pt]
\max\left\{g_r(n,t),g_r\left(n,\frac{n-2}{2}\right)+\frac n2-1\right\},
& n\text{ even}.
\end{cases}
$$
For odd $n$ the bound is attained by the corresponding
$\Hcal^{(1)}_{n,t}$ constructions.
\end{theorem}

When $n$ is even, the additional term
$g_r(n,(n-2)/2)+n/2-1$ comes from the exceptional last condition in
Salia's degree-sequence theorem (Theorem \ref{thm:salia}).  We only use it as an upper-bound term.
We do not claim here that this endpoint value is always attained.

If $f$ is a function on an
integer interval $I$, then the first minimizer of $f$ on $I$ is the
smallest integer $x\in I$ for which
$$
f(x)=\min_{y\in I} f(y).
$$
Thus, when we speak about the first minimizer of $g_r(n,x)$ on an
interval, the variable $x$ is restricted to the integers in that interval.
\begin{theorem}[Stability]\label{thm:stability}
Fix $r\ge 4$.  There are constants $C_r$ and $n_0=n_0(r)$ such that the
following holds for all $n\ge n_0$.  Let $k$ and $t=t(k)$ satisfy
$$
        t\ge C_r,\qquad
        \binom{t-1}{r-1}< k\le \binom{t}{r-1},
$$
and let $t_0$ be the first minimizer of $g_r(n,x)$ on
$r+1\le x\le \lfloor (n-1)/2\rfloor$.  Assume $t+1\le t_0$ and
$$
\begin{cases}
g_r(n,t+1)\ge g_r\left(n,\left\lfloor\frac{n-1}{2}\right\rfloor\right),
& n\text{ odd},\\[4pt]
g_r(n,t+1)\ge g_r\left(n,\frac{n-2}{2}\right)+\frac n2-1,
& n\text{ even}.
\end{cases}
$$
If $\Hcal$ is an $n$-vertex $r$-uniform hypergraph with
$\delta(\Hcal)\ge k$, with no Berge Hamiltonian cycle, and with
$$
        e(\Hcal)>g_r(n,t+1),
$$
then $\Hcal$ is a subhypergraph of either $\Hcal^{(1)}_{n,t}$ or
$\Hcal^{(2)}_{n,t}$.
\end{theorem}

The rest of the paper is organized as follows.  In
Section~\ref{sec:preliminaries} we collect notation and lemmas used
throughout the proof.  Section~\ref{sec:extremal-bound} proves the extremal
upper bound.  Section~\ref{sec:stability} proves the stability theorem.
Section~\ref{sec:concluding} contains concluding remarks and conjectures.

\section{Preliminaries}\label{sec:preliminaries}

We collect the notation used in the rest of the paper.  All hypergraphs
are finite, simple, and $r$-uniform unless explicitly stated otherwise.
For a hypergraph $\Hcal$, write $V(\Hcal)$ and $E(\Hcal)$ for its vertex
set and edge set, and write $d_{\Hcal}(v)$ for the degree of a vertex
$v$.  When the hypergraph is clear from the context, we write simply
$d(v)$.  For a vertex set $X\subseteq V(\Hcal)$, write $\Hcal[X]$ for the
subhypergraph induced by $X$.  The shadow graph of $\Hcal$, denoted by
$\partial\Hcal$, is the graph on $V(\Hcal)$ in which two vertices are
adjacent if they are contained together in some hyperedge of $\Hcal$.
For a graph $G$, write $N_G(v)$ and $d_G(v)$ for the neighborhood and
degree of $v$, and write $G[X]$ for the subgraph induced by $X$. 

\begin{theorem}[Salia \cite{Salia_2024}]\label{thm:salia}
Let $r\ge 3$ and $n>2r$.  Let
$d_1\le d_2\le \cdots\le d_n$ be the degree sequence of an
$n$-vertex $r$-uniform hypergraph $\Hcal$.  If
$$
\begin{aligned}
&d_i>i &&\text{for }1\le i<r,\\
&d_i>\binom{i}{r-1} &&\text{for }r\le i\le \left\lfloor\frac{n-1}{2}\right\rfloor,\\
&d_{(n-2)/2}>\binom{(n-2)/2}{r-1}+1 &&\text{when }n\text{ is even},
\end{aligned}
$$
then $\Hcal$ has a Berge Hamiltonian cycle.
\end{theorem}

The extremal expression in Theorem~\ref{thm:extremal} is governed by the
one-variable function $g_r(n,t)$ from \eqref{eq:def-gr}.  We shall
repeatedly use the fact that, on the range relevant to Salia's theorem, this
function has only one valley.
Thus, whenever $t$ is fixed by the minimum-degree condition, the largest
possible value among all later low-degree indices occurs at an endpoint.

\begin{lemma}\label{lem:g-unimodal}
The function $g_r(n,t)$ is decreasing and then increasing as $t$
increases.  If $t_0$ is its first minimizer on
$r+1\le t\le \lfloor (n-1)/2\rfloor$, then
$$
        \frac{t_0}{n}\to
        \alpha_r:=\frac{1}{1+r^{1/(r-1)}}<\frac12 .
$$
Consequently, for every fixed $r$ there is $\varepsilon=\varepsilon(r)>0$
such that $t_0\le n/2-\varepsilon n$ for all sufficiently large $n$.
\end{lemma}

\begin{proof}
Let $\Delta_t=g_r(n,t)-g_r(n,t-1)$.  Then
$$
\begin{aligned}
\Delta_t
&=\binom{n-t}{r}-\binom{n-t+1}{r}
  +t\binom{t}{r-1}-(t-1)\binom{t-1}{r-1}\\
&=-\binom{n-t}{r-1}
  +\binom{t}{r-1}+(t-1)\binom{t-1}{r-2}.
\end{aligned}
$$
Moreover,
$$
\begin{aligned}
\Delta_{t+1}-\Delta_t
&=\binom{n-t-1}{r-2}
  +\binom{t}{r-2}
  +t\binom{t}{r-2}
  -(t-1)\binom{t-1}{r-2}\\
&=\binom{n-t-1}{r-2}
  +2\binom{t}{r-2}
  +(t-1)\binom{t-1}{r-3}>0 .
\end{aligned}
$$
Thus the sequence $\Delta_t$ is strictly increasing, and so $g_r(n,t)$ is
unimodal.

Put $t=\lfloor xn\rfloor$, where $0<x<1$ is fixed.  Dividing the formula
for $\Delta_t$ by $n^{r-1}/(r-1)!$ gives
$$
        \frac{(r-1)!\Delta_t}{n^{r-1}}
        =r x^{r-1}-(1-x)^{r-1}+o(1).
$$
The function $r x^{r-1}-(1-x)^{r-1}$ is strictly increasing on $(0,1)$
and has the unique root
$$
        \alpha_r=\frac{1}{1+r^{1/(r-1)}}.
$$
For every $\eta>0$, the displayed asymptotic gives
$\Delta_{\lfloor(\alpha_r-\eta)n\rfloor}<0$ and
$\Delta_{\lceil(\alpha_r+\eta)n\rceil}>0$ for all sufficiently large
$n$.  Since $\Delta_t$ is increasing, the first minimizer lies between
these two indices.  Hence $t_0/n\to \alpha_r$.  As $\alpha_r<1/2$, the
last assertion follows.
\end{proof}

\section{The extremal bound}\label{sec:extremal-bound}

\begin{proof}[Proof of Theorem~\ref{thm:extremal}]
Let $d_1\le\cdots\le d_n$ be the degree sequence of $\Hcal$.
Since $\delta(\Hcal)\ge r+1$, the first line in Salia's theorem cannot
fail.  As $\Hcal$ is not Berge Hamiltonian, one of the remaining
conditions in Theorem~\ref{thm:salia} must fail.

Assume first that $n$ is odd.  Then there is an index
$i$ with
$$
        r\le i\le \left\lfloor\frac{n-1}{2}\right\rfloor
        \quad\text{and}\quad
        d_i\le \binom{i}{r-1}.
$$
The minimum-degree assumption gives $i\ge t$.  Let $X$ be the set of the
$i$ vertices of smallest degree.  Every edge either lies in $V(\Hcal)\setminus X$
or meets $X$, and the number of edges meeting $X$ is at most
$\sum_{x\in X}d(x)\le i\binom{i}{r-1}$.  Therefore
$$
        e(\Hcal)\le \binom{n-i}{r}+i\binom{i}{r-1}=g_r(n,i).
$$
By Lemma~\ref{lem:g-unimodal}, the maximum of $g_r(n,i)$ on
$t\le i\le \lfloor(n-1)/2\rfloor$ is attained at one of the endpoints.
This proves the odd case.

Now assume that $n$ is even.  If some
$r\le i\le (n-2)/2$ satisfies $d_i\le \binom{i}{r-1}$, the same argument
gives $e(\Hcal)\le g_r(n,i)$ with $i\ge t$.  Otherwise the only possible
failure of Theorem~\ref{thm:salia} is
$$
        d_{(n-2)/2}\le \binom{(n-2)/2}{r-1}+1.
$$
With $i=(n-2)/2$, the preceding counting gives
$$
        e(\Hcal)
        \le \binom{n-i}{r}
           +i\left(\binom{i}{r-1}+1\right)
        =g_r\left(n,\frac{n-2}{2}\right)+\frac n2-1.
$$
Together with Lemma~\ref{lem:g-unimodal}, this proves the even case.

For odd $n$, the constructions $\Hcal^{(1)}_{n,t}$ and
$\Hcal^{(1)}_{n,\lfloor(n-1)/2\rfloor}$ have minimum degree
$\binom{t}{r-1}$ and
$\binom{\lfloor(n-1)/2\rfloor}{r-1}$, respectively, contain no Berge
Hamiltonian cycle, and have the asserted numbers of edges.  Thus the
bound is sharp in the odd case.
\end{proof}

\section{Stability}\label{sec:stability}

Throughout this section, let the parameters satisfy the assumptions of
Theorem~\ref{thm:stability}.  The constant $C_r$ will be determined by
the finite inequalities that appear in the proof.  We may assume that
$\Hcal$ is edge-maximal subject to having no Berge Hamiltonian cycle,
since adding edges preserves the minimum-degree and edge-count
assumptions, and containment in $\Hcal^{(1)}_{n,t}$ or
$\Hcal^{(2)}_{n,t}$ passes to subhypergraphs.
Let
$$
        d_1\le d_2\le \cdots\le d_n
$$
be the degree sequence, and write $v_i$ for a vertex of degree $d_i$.

\begin{lemma}\label{lem:low-set}
Let $T=\{v_1,\ldots,v_t\}$ and $A=V(\Hcal)\setminus T$.  Then every
$x\in T$ satisfies
$$
        \binom{t-1}{r-1}+1\le k\le d(x)\le \binom{t}{r-1}.
$$
\end{lemma}

\begin{proof}
The lower bound follows from $\delta(\Hcal)\ge k$ and the definition of
$t$.  We prove the upper bound.

By Theorem~\ref{thm:salia}, some Salia inequality must fail.  In the odd
case this gives an index $i$ with
$r\le i\le \lfloor(n-1)/2\rfloor$ and
$d_i\le \binom{i}{r-1}$.  As in the proof of
Theorem~\ref{thm:extremal}, this implies
$e(\Hcal)\le g_r(n,i)$, and the minimum-degree condition gives $i\ge t$.
If $i\ge t+1$, then the assumptions $t+1\le t_0$ and
$g_r(n,t+1)\ge g_r(n,\lfloor(n-1)/2\rfloor)$ imply, by unimodality, that
$g_r(n,i)\le g_r(n,t+1)$, contradicting
$e(\Hcal)>g_r(n,t+1)$.  Hence $i=t$, and so $d_t\le\binom{t}{r-1}$.

The even case is identical, except that the endpoint
$g_r(n,(n-2)/2)+n/2-1$ is used when the last even Salia condition fails.
The corresponding endpoint is dominated by $g_r(n,t+1)$ by assumption,
so this case is also impossible.  Thus again $d_t\le\binom{t}{r-1}$.
\end{proof}

The next calculation is used repeatedly.  Let
$$
        m_A:=\left|\binom{A}{r}\setminus E(\Hcal)\right|
$$
be the number of missing $r$-sets inside $A$, and let $e_T$ denote the
number of hyperedges of $\Hcal$ meeting $T$.  By
Lemma~\ref{lem:low-set},
$$
        e_T\le \sum_{x\in T}d(x)\le t\binom{t}{r-1}.
$$
Since $e(\Hcal)>g_r(n,t+1)$, it follows that
$$
\begin{aligned}
 m_A
 &=\binom{n-t}{r}-e(\Hcal[A])\\
 &=\binom{n-t}{r}-e(\Hcal)+e_T\\
 &<\binom{n-t}{r}-g_r(n,t+1)+t\binom{t}{r-1}\\
 &=g_r(n,t)-g_r(n,t+1)\\
 &=\binom{n-t-1}{r-1}
   -\binom{t}{r-1}
   -(t+1)\binom{t}{r-2}.
\end{aligned}
$$
Thus $m_A<M$, where
\begin{equation}\label{eq:missing-A}
M:=
\binom{n-t-1}{r-1}
-\binom{t}{r-1}
-(t+1)\binom{t}{r-2}.
\end{equation}

\begin{lemma}\label{lem:t-large-ineq}
For every fixed $r\ge4$ there is a constant $C_r^{(1)}$ such that, for
all $t\ge C_r^{(1)}$,
$$
        \binom{t+1}{r-1}+(t+1)
        <
        \binom{t}{r-1}+(t+1)\binom{t}{r-2}.
$$
\end{lemma}

\begin{proof}
Using $\binom{t+1}{r-1}=\binom{t}{r-1}+\binom{t}{r-2}$, the difference
between the right-hand side and the left-hand side is
$$
        t\binom{t}{r-2}-(t+1).
$$
Since $r\ge4$, we have $r-2\ge2$, and hence
$t\binom{t}{r-2}$ grows at least cubically in $t$.  Therefore it is
larger than $t+1$ for all sufficiently large $t$, where the threshold
depends only on $r$.
\end{proof}

\begin{lemma}\label{lem:A-complete}
The induced hypergraph $\Hcal[A]$ is complete.
\end{lemma}

\begin{proof}
%We give the saturation-and-rotation argument in full.  
Our proof relies on the assumed maximality of $\Hcal$ and a P\'osa-type rotation argument.
Since
$t+1\le t_0$, Lemma~\ref{lem:g-unimodal} gives
$t\le n/2-\varepsilon n$ for some $\varepsilon=\varepsilon(r)>0$ and all
sufficiently large $n$.  This linear separation from $n/2$ follows from
the position of the first minimizer and does not impose any lower bound on
$t$.

For $a\in A$, let $d_A(a)$ denote the degree of $a$ in $\Hcal[A]$, and
put
$$
 A_1=\left\{a\in A:\ d_A(a)\ge
        \binom{\lceil n/2\rceil}{r-1}+n\right\},
        \qquad
 A_2=A\setminus A_1.
$$
Every vertex of $A_2$ is contained in more than
$$
        D:=\binom{n-t-1}{r-1}
        -\binom{\lceil n/2\rceil}{r-1}-n
$$
missing $r$-sets of $A$.  We first prove that $|A_2|\le 1$.  If
$u,v\in A_2$ are distinct, then the union of the missing $r$-sets of
$A$ containing $u$ and the missing $r$-sets of $A$ containing $v$ has size
larger than
$$
        2D-\binom{n-t-2}{r-2},
$$
because the overlap consists only of missing $r$-sets containing both
$u$ and $v$.  We claim that this is larger than the upper bound $M$ in
\eqref{eq:missing-A}, for all sufficiently large $n$.  Indeed, after
multiplication by $(r-1)!/n^{r-1}$ and writing $x=t/n$, the difference
between the left side and $M$ is
$$
        (1-x)^{r-1}+r x^{r-1}-2^{2-r}+o(1),
$$
uniformly for $0\le x\le 1/2$.  The function
$(1-x)^{r-1}+r x^{r-1}-2^{2-r}$ is positive on $[0,1/2]$: the sum
$(1-x)^{r-1}+x^{r-1}$ is at least $2^{2-r}$, and the additional
$(r-1)x^{r-1}$ makes the inequality strict, with a positive minimum on
the compact interval.  Hence two vertices in $A_2$ would create more
missing edges in $A$ than allowed by \eqref{eq:missing-A}.  Thus
$|A_2|\le1$.

We first show that there is no Berge Hamiltonian path in $\Hcal$ whose two
endvertices both lie in $A_1$.  Suppose that $P$ is such a path, with
endvertices $x,y\in A_1$.  At each of $x$ and $y$ there are at least
$\binom{\lceil n/2\rceil}{r-1}$ hyperedges of $\Hcal[A]$ containing the
endpoint which are not defining hyperedges of $P$.  Hence each endpoint
has at least $\lceil n/2\rceil$ neighbors in $A$ through non-defining
hyperedges.  The usual P\'osa rotation argument applies: if $z$ is joined
to $y$ by a non-defining hyperedge and the successor $z^+$ of $z$ on $P$
is joined to $x$ by a non-defining hyperedge, then these two hyperedges
together with the path $P$ form a Berge Hamiltonian cycle.  Thus the set
of successors of the non-defining neighbors of $y$ is disjoint from the
non-defining neighborhood of $x$.  Both sets have size at least
$\lceil n/2\rceil$, impossible on an $n$-vertex path.  Therefore, no such
path exists.

It follows that $\Hcal[A_1]$ is complete.  If an $r$-set
$R\subseteq A_1$ were missing, then by edge-maximality
$\Hcal+R$ would contain a Berge Hamiltonian cycle.  The new edge $R$ must
be used in that cycle; deleting it gives a Berge Hamiltonian path in
$\Hcal$ whose two endpoints lie in $R\subseteq A_1$, a contradiction.

It remains to prove that $A_2=\emptyset$.  Suppose to the contrary that
$A_2=\{u\}$.  Since
$|A_1|=n-t-1\ge n/2+\varepsilon n-1$, if every $r$-set consisting of
$u$ and $r-1$ vertices of $A_1$ were present, then
$$
        d_A(u)\ge \binom{|A_1|}{r-1}>
        \binom{\lceil n/2\rceil}{r-1}+n,
$$
for all sufficiently large $n$, contrary to $u\in A_2$.  Thus there is a
missing edge $R=\{u\}\cup S$ with $S\in\binom{A_1}{r-1}$.

By maximality, $\Hcal+R$ has a Berge Hamiltonian cycle.  Removing $R$
gives a Berge Hamiltonian path $P$ in $\Hcal$.  Since no Berge Hamiltonian
path has both endpoints in $A_1$, the endpoints of $P$ are $u$ and some
$x\in S\subseteq A_1$.  Orient $P$ from $x$ to $u$.

Let $B$ be the set of vertices $w\in A_1$ for which there is a
non-defining hyperedge of $P$ contained in $A$ and containing both $u$ and
$w$.  If $w\in B$, then the successor $w^+$ of $w$ on $P$ is not in
$A_1$; otherwise replacing the defining edge of $P$ between $w$ and
$w^+$ by a non-defining hyperedge containing $u,w$ gives a Berge
Hamiltonian path in $\Hcal$ with endpoints $x$ and $w^+$, both in $A_1$,
which is impossible.  The successor map on $P$ is injective, and
$V(\Hcal)\setminus A_1=T\cup\{u\}$.  Hence
$$
        |B|\le t+1.
$$
All non-defining edges of $\Hcal[A]$ containing $u$ are contained in
$\{u\}\cup B$.  We also need to count the defining hyperedges of $P$
which lie in $\Hcal[A]$ and contain $u$.  If such a defining hyperedge
$E_i$ is used for the pair $p_i p_{i+1}$ and $p_{i+1}\in A_1$, then
using $E_i$ instead on the pair $p_i u$ rotates $P$ into a Berge
Hamiltonian path with endpoints $x$ and $p_{i+1}$, both in $A_1$, a
contradiction.  Hence the successor $p_{i+1}$ of every such defining
edge, except possibly the last edge incident with the endpoint $u$, lies
in $T$.  The successor map is injective, so there are at most $t+1$
defining hyperedges of $\Hcal[A]$ that contain $u$.  Therefore
$$
        d_A(u)\le \binom{t+1}{r-1}+(t+1).
$$
Consequently the number of missing $r$-sets of $A$ containing $u$ is at
least
$$
        \binom{n-t-1}{r-1}-\binom{t+1}{r-1}-(t+1).
$$
By Lemma~\ref{lem:t-large-ineq}, if $t\ge C_r^{(1)}$, then
the displayed lower bound is larger than $M$.
Thus the number of missing $r$-sets of $A$ containing $u$ is larger than
$M$, contradicting \eqref{eq:missing-A}.  Hence $A_2=\emptyset$, and
therefore $A=A_1$ and $\Hcal[A]$ is complete.
\end{proof}

Let $G=\partial\Hcal$ be the shadow graph.  An edge $xy$ of $G$ is called
heavy if it is contained in at least $2t$ hyperedges of $\Hcal$.
Let $G_T$ be the graph obtained from $G$ by deleting all edges inside
$A$, and let $G_h$ be the subgraph of $G_T$ formed by the heavy edges.

\begin{lemma}\label{lem:heavy-degree}
Every vertex of $T$ has degree at least $t$ in $G_T$ and degree at least
$t-1$ in $G_h$.
\end{lemma}

\begin{proof}
Fix $x\in T$.  By Lemma~\ref{lem:low-set}, $d(x)\le\binom{t}{r-1}$.
Since $\Hcal[A]$ is complete and $|A|>t$, there is an $(r-1)$-set
$R\subseteq A$ such that $\{x\}\cup R$ is not an edge of $\Hcal$.
Adding this edge creates a Berge Hamiltonian cycle, and deleting the new
edge from that cycle gives a Berge Hamiltonian path $P$ from some
vertex $a\in A$ to $x$.  The proof of Lemma~\ref{lem:A-complete} showed
that no Berge Hamiltonian path of $\Hcal$ has both endpoints in $A$,
which is why one endpoint of $P$ is $x$.

If $x$ had at most $t-1$ neighbors in $G_T$, then every hyperedge
containing $x$ would be contained in $\{x\}\cup N_{G_T}(x)$, and hence
$$
        d(x)\le \binom{t-1}{r-1},
$$
contrary to $d(x)\ge k\ge \binom{t-1}{r-1}+1$.  Thus
$d_{G_T}(x)\ge t$.

Orient $P$ from $a$ to $x$.  We shall use the following rotation
observation.  If $y$ is joined to $x$ by a hyperedge not used as a
defining hyperedge of $P$, and if the successor $y^+$ of $y$ on $P$ lies
in $A$, then, because $\Hcal[A]$ is complete and $n$ is large, there is a
hyperedge of $\Hcal[A]$ containing $a$ and $y^+$ which is not a defining
hyperedge of $P$.  Together with the non-defining hyperedge containing
$x$ and $y$, this rotates $P$ into a Berge Hamiltonian cycle, a
contradiction.  Hence every such successor $y^+$ lies in $T$.

Similarly, consider a defining hyperedge $E_i$ of $P$ which contains
$x$.  If $E_i$ is used for the pair $p_i p_{i+1}$ and
$p_{i+1}\in A$, then using the same hyperedge $E_i$ for the pair
$p_i x$ rotates $P$ into a Berge Hamiltonian path with endpoints
$a$ and $p_{i+1}$, both in $A$, again impossible.  Hence, except for the
last defining edge incident with the endpoint $x$, the successor
$p_{i+1}$ must lie in $T$.  Thus at most $t+1$ defining hyperedges of
$P$ contain $x$.

Now suppose $x$ has at most $t-2$ heavy neighbors in $G_T$.  All
hyperedges containing $x$ and using only heavy neighbors contribute at
most $\binom{t-2}{r-1}$ choices.  Hyperedges containing at least one
non-heavy pair $xy$ are charged to such a pair; each non-heavy pair is
contained in fewer than $2t$ hyperedges.  The rotation observation above
implies that, apart from the defining hyperedges containing $x$, at most
two non-heavy pairs at $x$ can occur without creating a Berge Hamiltonian
cycle.  Hence the non-heavy contribution is at most
$2(2t-1)$, and the defining contribution is at most $t+1$.  Therefore
$$
        d(x)\le \binom{t-2}{r-1}+2(2t-1)+(t+1).
$$
For $r\ge4$ and $t$ larger than a constant $C_r^{(2)}$, this is strictly smaller than
$\binom{t-1}{r-1}+1$, contradicting Lemma~\ref{lem:low-set}.
Therefore, $x$ has at least $t-1$ heavy neighbors.
\end{proof}

The following graph lemma is the technical heart of the stability proof.
It is a finite shadow statement; the constants are chosen so that a greedy
choice of distinct hyperedges is possible afterwards.

\begin{lemma}[Shadow covering lemma]\label{lem:shadow}
Let $G_T$ and $G_h$ be as above.  Suppose every vertex of $T$ has degree
at least $t$ in $G_T$ and degree at least $t-1$ in $G_h$.  Then one of the
following holds.
\begin{enumerate}[label=(\roman*)]
\item $G_T$ contains a linear forest $F$ covering all vertices of $T$,
whose components have endpoints in $A$ and have no edges inside $A$.
Moreover, to the edges of $F$ not in $G_h$ 
%may be assigned in advancedistinct hyperedges  
we can assign distinct hyperedges containing them.
\item $|N_{G_T}(T)\cap A|\le 1$.
\item $G_T[T]$ is empty and $|N_{G_T}(T)\cap A|\le t$.
\end{enumerate}
\end{lemma}

\begin{proof}
    
    First we pick a linear forest $F_0$ in $G_T$ such that it consists of heavy edges inside $T$ with at most one exception as an end edge of a path component.
    Let $P_1,\dots, P_s$ be the paths in $F_0$ and assume the size of the paths is $p_1 \ge p_2 \ge \cdots \ge p_s$, such that $\sum_{i=1}^{s}p_i = t$.
    We choose $F_0$ such that $(p_1,\ldots,p_s)$ is maximized in lexicographic order.
    Subject to this, we pick $F_0$ without a light edge, if possible. If not, then there is some $i$ such that $P_i$ contains the light edge. We pick $F_0$ such that $i$ is maximized.
    %if there is a light edge in $P_i$, we pick the $F_0$ such that $i$ is maximized~(when there is no light edge, we make a deal that $i = \infty$).
By the way how we picked $F_0$, if there is a light edge in $P_i$, then $p_j < p_i$ for all $j > i$.

    \noindent
    \textbf{Case 1:} $s = t$ and $p_i = 1, i=1,2,\ldots,t$.
    
    In this case, $G_T[T]$ has no edge.
    Let $F'$ be a linear forest with endpoints in $A$, covering the maximum number of vertices in $T$, and for each covered $v_i$, at most one neighbor is light.
    If $F'$ covers all the vertices in $T$, then we are done and (i) holds, because every light pair is contained in a distinct hyperedge since $G_T[T]$ has no edge.
    
    Suppose $v_{t-1}, v_{t}$ are not covered by $F'$.
    Note that there are at least $t-1$ heavy neighbors of $v_{t-1}$ in $A$ and at least $t$ neighbors of $v_{t-1}$ in $A$.
    There are at most $t-3$ internal vertices in $A$ used in $F'$. Then we can pick a heavy neighbor of $v_{t-1}$ in $T$ avoiding the internal vertices.
    If a vertex in $F'$ is chosen, then we can pick another neighbor of $v_{t-1}$ in $A$ avoiding the internal vertices and the other endpoint of the path. This way we obtain a linear forest with endpoints in $A$, covering more vertices in $T$ than $F'$, a contradiction to the choice of $F'$.
    Otherwise, we can pick another arbitrary neighbor of $v_{t-1}$ in $A$ avoiding the internal vertices, again obtaining a linear forest with endpoints in $A$, covering more vertices in $T$ than $F'$, a contradiction to the choice of $F'$.

    Now suppose only $v_t$ is not covered.
    If $F'$ is not a single path, then there are at most $t-3$ internal vertices in $A$.
    Then we can complete the proof by a similar argument as before.
    When $F'$ contains exactly one path, there are $t-2$ internal vertices in $A$.
    The above argument only fails when the neighbor set of $v_t$ is exactly $F'\cap A$.
    Then we have a cycle in $G_T$ covering $T$ such that every $v_i$ is incident to at most one light edge.
    Assume the cycle is $v_1a_1v_2a_2\cdots v_ta_tv_1$.
    Let $A' = F' \cap A = \{a_1,a_2,\ldots,a_t\}$, then we claim that $N_{G_T}(T) = A'$ and thus we arrive to (iii). Indeed, if $v_i$ has a neighbor outside $A'$, say $a'$, without loss of generality, assume $v_ia_i$ is heavy, then we have a path $a'v_ia_i\cdots v_1a_1 \cdots v_{i-1}a_{i-1}$, a contradiction to the choice of $F'$.

     \noindent
    \textbf{Case 2:} $s=1, p_1 = t$.
    Assume the path is $v_1v_2\cdots v_t$.
    
    \textbf{Subcase 2.1:} all the edges are heavy in $F_0$. %Dani: in $F_0$??? or inside $T$???...
    
    If there is no heavy Hamiltonian cycle in $T$, then $v_1$ has a heavy neighbor $a_1 \in A$, and $v_t$ has a heavy neighbor $a_2 \neq a_1 \in A$ which leads to (i).
    Without loss of generality, assume $v_1v_2\ldots v_tv_1$ is a heavy cycle.
    
    If $v_1$ has a heavy neighbor $a_1 \in A$, 
    then we claim that for every $v_i$ with $i\neq 1$, $N(v_i)=T\cup \{a_1\}$.
    Otherwise, suppose $v_i$~($i < t$) has a neighbor $a_i \neq a_1$, then consider the heavy neighbors of $v_2$.
    If $v_2$ has a heavy neighbor $a_2 \neq a_1 \in A$, then (i) holds.
    If $v_{i+1}$ is a heavy neighbor of $v_2$, then we are done as well, because there is a heavy path with $v_1$ and $v_i$ as end vertices, and they have distinct heavy neighbors in $A$.
    Thus, according to 
    Lemma \ref{lem:heavy-degree} for vertex $v_2$, $a_1$ must be a heavy neighbor of $v_2$.
    We can continue this process and take $v_j$ as the beginning vertex for $j=2,\dots,i-1$ one by one, and we can get that $a_1$ is a heavy neighbor of $v_{i-1}$. This way we also arrive to (i).

    If $v_1$ has a heavy neighbor $a_1 \in A$, then by the previous analysis, $N(v_i) = T \cup \{a_1\}$ for all $i \neq 1$.
    If $v_1$ has no neighbors in $A\setminus \{a_1\}$, then we are done as (ii) holds.
    Otherwise, $v_1$ has a neighbor other than $a_1$ in $A$, then the heavy neighbors of each $v_i$ must be all the vertices in $T$.
    Then $G_h[T]$ is a clique.
    Note that if $v_1$ has no heavy neighbor in $A$, by the symmetry, $G_h[T]$ is a clique.
    So this is the only remaining case.

    When $G_h[T]$ is a clique, it is easy to prove that $N_{G_T}(T) \cap A$ has at most one vertex, thus (ii) holds.

    \textbf{Subcase 2.2:} there is a light edge in $F_1$, say $v_{t-1}v_{t}$. 

    We may assume $v_1v_t$ is not a heavy edge, otherwise $v_tv_1v_2\dots v_{t-1}$ is a path with both edges heavy, and we are done by the previous case.
    Since $v_1v_t$ is not heavy, $v_1$ has a heavy neighbor $a_1 \in A$.
    Note that $v_{t-1}v_{t}$ is not a heavy edge, then $v_{t}$ has at least two heavy neighbors in $A$.
    Then we are done as (i) holds.

    \noindent
    \textbf{Case 3:} $p_2 = 1$.
    
    Assume the paths are $P_1=v_1 \ldots v_{p_1}$ and the singletons $u_1, u_2,\ldots, u_q$ with $p_1 + q = t$ and $q\geq 1$.

    We shall repeatedly use the following simple observation. Let $F$ be a
linear forest whose components have endpoints in $A$, and let
$$
I(F)=\{a\in A\cap V(F): d_F(a)=2\}
$$
be the set of internal $A$-vertices of $F$. Suppose that
$u\in T\setminus V(F)$ has at least $|I(F)|+3$ neighbors in $A$.
Then $u$ can be inserted into $F$ using two of these neighbors without
creating a cycle. Indeed, after avoiding $I(F)$, at least three neighbors
of $u$ remain in $A$. Among any three such vertices, two are not the two
endpoints of the same component of $F$. Choosing such two vertices $a,b$
and adding the edges $au$ and $ub$ keeps a linear forest whose components
still have endpoints in $A$. The same statement holds with ``neighbors''
replaced by ``heavy neighbors''.

    \textbf{Subcase 3.1:} there is no light edge in $P_1$.
    
    Then each $u_i$ has at least $q+1$ heavy neighbors in $A$ and at least $q+2$ neighbors in $A$.
    $v_1, v_{p_1}$ has at least $q+1\geq 2$ heavy neighbors in $A$ and at least $q+2$ neighbors in $A$.

    First pick a heavy neighbor $a_1$ of $v_1$ in $A$, and a neighbor $a_2$ of $v_{p_1}$ in $A$ with $a_2 \neq a_1$.
    Let $F'$ be the linear forest with endpoints in $A$, covering the maximum number of vertices in $T$, containing $a_1P_1a_2$ and for each $u_i$, and there is at most one light edge in $F'$ incident to $u_i$ for some $i=1,\dots,q$.
    
    We define $I(F')$ to denote the vertices in $A\cap V(F')$ with degree two in $F'$, which is the set of internal vertices in $F'$.

    If $F'$ covers all the vertices in $T$, then (i) holds.
    If a vertex, say $u_q$ is not covered, then notice that $|I(F')|\leq q-2$, then we can pick a heavy neighbor $a_q$ of $u_q$ in $A$ avoiding $I(F')$.
    Since $q+2-(q-1) > 3$, we can pick another heavy neighbor of $u_q$ in $A$ avoiding $I(F')$. In the case these two neighbors are endpoints of the same path in $F'$, we can pick a third heavy neighbor. 
    %and when $u_q$ is in $F'$ avoiding the other endpoint of the path in $F'$ containing $u_q$ (which avoids the case when this extends to a cycle), a contradiction to the choice of $F'$.

    \textbf{Subcase 3.2:} $v_1v_2$ is a light edge.
    
    Then $v_1$ has at least one heavy neighbor $a_1$ in $A$.
    Note that $v_{p_1}v_1$ cannot be a heavy edge (otherwise $v_2v_3 \dots v_{p_1}v_1$ is a heavy path, which is done by the previous case), then $v_{p_1}$ has at least $q+1$ heavy neighbors in $A$.
    Pick a heavy neighbor $a_2$ of $v_{p_1}$ in $A$ with $a_2 \neq a_1$.

    For each $u_i$, when $p_1 = 2$, it cannot have $v_1$ and $v_2$ as heavy neighbors, because it will cause a heavy $P_2$, which violates the choice of $F_0$.
    When $p_1 > 2$, it cannot have $v_2$ or $v_{p_1}$ as heavy neighbors, or we can find a heavy path with length $p_1$, which also violates the choice of $F_0$.
    Then $u_i$ has at least $q+1$ heavy neighbors in $A$.

Now let us extend $u_1,\ldots,u_{q-1}$ one by one. Suppose that, before
$u_i$ is inserted, the current forest is $F_i$. Since
$|I(F_i)|\le i-1$ and $u_i$ has at least $q+1$ heavy neighbors in
$A$, for $i<q$ we have
$$
q+1 \ge |I(F_i)|+3 .
$$
By the observation above, we can choose two heavy neighbors of $u_i$ in
$A\setminus I(F_i)$ which are not the two endpoints of the same component
of $F_i$. Adding the two corresponding edges inserts $u_i$ and keeps a
linear forest with endpoints in $A$.

    When we arrive at $u_q$, let $F'$ be the current linear forest.
    If $p_1 > 2$, it fails only when $u_q$ has $v_1, v_3,v_4,\ldots, v_{p_1-1}$ as heavy neighbors.
    Since all $u_i$'s are equivalent, all the $u_i$'s have $v_1, v_3,\ldots, v_{p_1-1}$ as heavy neighbors.
    When $q > 2$ and $p_1 = 3$, we have a heavy path $u_1v_1u_2$, a contradiction.
    When $q \ge 1$ and $p_1 > 3$, we have a longer heavy path $v_2v_3\ldots v_{p_1-1}u_qv_1$, a contradiction.
    Note that it is impossible that $q = 1$ and $p_1 = 3$ since $t$ is large.

    If $p_1 = 2$, then $G_h[T]$ is empty and every vertex in $T$ has at least $t-1$ heavy neighbors in $A$.
    If there is a path of length two or a matching of size two in $G_T[T]$ such that the edges are from distinct hyperedges, then we can extend greedily to the case (i).
    So the hyperedge containing $v_1,v_2$ is the only hyperedge that contains at least two vertices in $T$.
    Without loss of generality, assume $u_q$ is not in the hyperedge since $t > r$.
    Then there are at least $t$ neighbors of $u_q$ in $A$.
    Now we can first pick a heavy neighbor $a_q$ of $u_q$ in $A$ avoiding the vertices in $I(F')$, and then pick a neighbor of $u_q$ in $A$ avoiding the internal vertices and the other endpoint of the path if $a_q$ is in $F'$, which leads to (i).

    \textbf{Case 4:} $s \ge 2$ and $2 \le p_1 \le t-2$.
    In this case, we would like to extend $F_0$ to a linear forest $F$ satisfying (i).
    To achieve this, we aim to extend every $P_i$ one by one using heavy neighbors.

    Start with $P_1$ with size $p_1$.
    If there is no light edge in $P_1$, then for each end point of $P_1$, the vertices in $T\setminus P_1$ cannot be heavy neighbors, so there are at least $t - p_1$ heavy neighbors in $A$.
    Since $t - p_1 \ge 2$, we can extend $P_1$ to a path $P_1'$ with both end points in $A$ using heavy neighbors.
    If there is a light edge in $P_1$, assume that $P_1 = v_1v_2\ldots v_{p_1}$ and $v_{1}v_{2}$ are light.
    Similarly, $v_1$ has at least one heavy neighbor in $A$ and $v_{p_1}$ has at least $t - p_1 \ge 2$ heavy neighbors in $A$, thus we can extend $P_1$ to a path $P_1'$ with both end points in $A$ using heavy neighbors.

    Assume that after $i$ paths in $F_0$, we arrive at a path $P_{i+1}$ with $p_{i+1} \ge 2$. 
    Let $u_1,u_2$ be the endpoints of $P_{i+1}$.
    If we haven't met a light edge before this moment, then $u_1$ and $u_2$ cannot have the previous endpoints of the paths in $F_0$ as their heavy neighbors.
    If there is a light edge in $P_{j}$ for some $j \le i$, then the internal vertex in the light edge cannot be a heavy neighbor of $u_1$ or $u_2$.
    Moreover, let $P_{i+1} = u_1u_2\ldots u_{p_{i+1}}$, then $u_{p_{i+1}}$ cannot have vertices in $P_{i+2}, \ldots, P_{s}$ as its heavy neighbors unless $u_{p_{i}}u_{p_{i+1}}$ is a light edge and $p_{i+1} \ge 3$.
    As a conclusion, one of the end points has at least $2i$ heavy neighbors in $A$ and the other has at least $2i + p_{i+2} + \ldots + p_{s}$ heavy neighbors in $A$.
    Let $F_i=P_1'\cup\dots \cup P'_{i}$, then $|I(F_i)\cap A|\leq i-1$.
    When $2i > i-1$ and $2i + p_{i+2} + \ldots + p_{s} > i-1 +2$, then we can extend $P_{i+1}$ to a path $P_{i+1}'$ with both end points in $A$ using heavy neighbors such that $P_1', \ldots, P_{i+1}'$ is a linear forest.
    It only fails when $i=1$ and $s=2$.
    In this case, $p_1 \ge 3$, otherwise we have $t \le 4$, a contradiction.
    Assume $P_1 = v_1v_2\ldots v_{p_1}$, then $v_2$ cannot be a heavy neighbor of $u_1$ or $u_2$ as well, so $u_1$ and $u_2$ both have at least three heavy neighbors in $A$ and we are done since (i) holds.
    
    Assume that after $p$ paths of length at least one and $q$ singletons in $F_0$, we arrive to a singleton $u$, and the current linear forest is $F'$. 
    Then we claim that $u$ has at least $2p+q$ heavy neighbors in $A$.
    If we haven't met a light edge before this moment, then the endpoints of these paths and the singletons, at least $2p+q$ vertices of $T$ are not heavy neighbors of $u$, thus $u$ has at least $2p+q$ heavy neighbors in $A$. 
    If there is a light edge in $P_{j}$ for some $j \le p$, let $P_j = u_1u_2\ldots u_{p_j}$ and $u_1u_2$ is light, then when $p_j = 2$, $u_1$ and $u_2$ cannot be a heavy neighbor of $u$, when $p_j > 2$, $u_2$ and $u_{p_j}$ cannot be a heavy neighbor of $u$.
    There are at most $p+q-1$ vertices in $I(F')$.
    Since $|I(F')|\le p+q-1$ and $p\ge 2$, the vertex $u$ has at least
$$
(2p+q)-(p+q-1)=p+1\ge 3
$$
heavy neighbors in $A\setminus I(F')$. By the observation above, two of
these neighbors can be chosen so that they are not the two endpoints of the
same component of $F'$. Adding the corresponding two heavy edges therefore
inserts $u$ without creating a cycle by the observation in Case 3, and the resulting graph is again a
linear forest whose components have endpoints in $A$.  
\end{proof}

\begin{proof}[Proof of Theorem~\ref{thm:stability}]
Let $T$ and $A$ be as in Lemma~\ref{lem:low-set}.  By
Lemma~\ref{lem:A-complete}, $\Hcal[A]$ is complete.  By
Lemma~\ref{lem:heavy-degree}, the hypotheses of the shadow covering Lemma~\ref{lem:shadow}
hold.

If Lemma~\ref{lem:shadow}(ii) holds, then every hyperedge meeting $T$ is
contained in $T\cup\{a\}$ for some single vertex $a\in A$, while all
edges inside $A$ are allowed.  Hence $\Hcal$ is a subhypergraph of
$\Hcal^{(2)}_{n,t}$.

If Lemma~\ref{lem:shadow}(iii) holds, then no hyperedge contains two
vertices of $T$, and all vertices of $A$ that occur together with a vertex
of $T$ lie in a set $S\subseteq A$ of size at most $t$.  Enlarging $S$ to
size $t$ if necessary, we see that $\Hcal$ is a subhypergraph of
$\Hcal^{(1)}_{n,t}$.

It remains to exclude Lemma~\ref{lem:shadow}(i).  Let $F$ be the linear
forest given there, with its already prescribed hyperedges for light
edges.  Because every remaining edge of $F$ is heavy and $F$ has fewer
than $2t$ edges incident with $T$, we may greedily choose distinct
hyperedges of $\Hcal$ for all edges of $F$.

Since $\Hcal[A]$ is complete, the shadow on $A$ is a complete graph.
Extend the linear forest $F$ to a Hamiltonian cycle $C$ of the shadow graph
by adding edges with both endpoints in $A$ and by inserting all unused
vertices of $A$ along these $A$-edges.  We now choose distinct hyperedges
for the added $A$-edges.  Each pair in $A$ is contained in
$$
        \binom{|A|-2}{r-2}
$$
edges of $\Hcal[A]$, and at most $n$ of these have already been used or
forbidden.  For $n$ sufficiently large this leaves at least $\binom r2$
available hyperedges for every added pair.  For any set $X$ of added
$A$-edges, the number of incidences between $X$ and available hyperedges
containing them is at least $\binom r2 |X|$, while any single
$r$-hyperedge contains at most $\binom r2$ pairs of $X$.  Hall's theorem
therefore gives a matching from the added $A$-edges to available
hyperedges.  Together with the already chosen hyperedges for $F$, this
turns $C$ into a Berge Hamiltonian cycle of $\Hcal$, a contradiction.

Thus, we have completed the proof. 
\end{proof}
\section{Concluding remarks}\label{sec:concluding}
In this paper we proved an extremal theorem and a stability theorem for
Berge Hamiltonian cycles in $r$-uniform hypergraphs under a minimum degree
condition. The extremal theorem gives an upper bound in terms of the function
$g_r(n,t)$ and Salia's degree-sequence theorem. The stability theorem shows
that, for $r \ge 4$, for sufficiently large $t$, and in the dense range before
the first minimizer of $g_r(n,x)$, every near-extremal non-Hamiltonian example
is contained in one of the two natural constructions $\mathcal{H}^{(1)}_{n,t}$ and
$\mathcal{H}^{(2)}_{n,t}$.

There are several natural directions in which the statement could be extended.
We formulate them as conjectures.

\begin{conjecture}
For fixed $r \ge 4$, the conclusion of
Theorem~\ref{thm:stability} remains true for all admissible values of $t$
satisfying $t+1\le t_0$, without the assumption that $t$ is sufficiently
large.
\end{conjecture}

The assumption $t \ge C_r$ is used in our proof to simplify several estimates
and to exclude small exceptional shadow configurations. We believe that these
exceptions are only technical, and that the same two constructions
$\mathcal{H}^{(1)}_{n,t}$ and
$\mathcal{H}^{(2)}_{n,t}$ should remain the only extremal obstructions
after finitely many small cases are checked separately.

\begin{conjecture}
The conclusion of Theorem~\ref{thm:stability} also holds for $r=3$.
\end{conjecture}

The present proof uses $r \ge 4$ when controlling the contribution of light
pairs. For $r=3$, the relevant gap is smaller and has the same order as the
error term coming from light pairs. Thus our argument does not directly apply,
but we expect that the same qualitative stability phenomenon should still be
true.

\begin{conjecture}
Let $r \ge 3$, let $n > 2r$ be even, and let $H$ be an $n$-vertex
$r$-uniform hypergraph with no Berge Hamiltonian cycle. Let
\[
r+1 \le k < \binom{(n-2)/2}{r-1},
\]
and let $t=t(k)$ be the unique integer satisfying
\[
\binom{t-1}{r-1} < k \le \binom{t}{r-1}.
\]
If $\delta(H) \ge k$, then
\[
e(H) \le
\max \left\{
g_r(n,t),\,
g_r\left(n,\frac{n-2}{2}\right)
\right\}.
\]
The two values in the maximum are attained by
$\mathcal{H}^{(1)}_{n,t}$ and
$\mathcal{H}^{(1)}_{n,(n-2)/2}$, respectively.
\end{conjecture}

The even case of Theorem~\ref{thm:extremal} contains the additional term
$g_r(n,(n-2)/2)+n/2-1$, which comes from the exceptional last condition in
Salia's degree-sequence theorem. The conjecture above says that this extra
term should be an artifact of the proof rather than a genuine extremal
phenomenon. In other words, we expect the even case to have the same extremal
picture as the odd case, with the endpoint $(n-2)/2$ replacing
$\lfloor (n-1)/2 \rfloor$.

\section*{Declaration on the Use of Generative AI}

The authors used ChatGPT to assist with literature searching and writing.
The mathematical statements, proofs, and final content were reviewed and
edited by the authors, who take full responsibility for the accuracy and
integrity of the paper.

\bigskip
\section*{Acknowledgments} 
The research of Wang is supported by the China Scholarship Council (No. 202506210200) and
the National Natural Science Foundation of China (Grant 12571372). 

The research of Gerbner is supported by the J\'anos Bolyai scholarship.

The research of Zhao is supported by the China Scholarship Council (No. 202506210250) and
the National Natural Science Foundation of China (Grant 12571372).

\bibliography{ref.bib}
\bibliographystyle{wyc4}

\end{document}